\input ./style/arxiv-vmsta.cfg
\documentclass[numbers,compress,v1.0.1]{vmsta}

\volume{3}
\issue{3}
\pubyear{2016}
\firstpage{237}
\lastpage{248}
\doi{10.15559/16-VMSTA62}

\usepackage[english]{babel}
\usepackage{enumerate}
\usepackage{mathbh}

\newtheorem{theo}{Theorem}

\theoremstyle{definition}

\theoremstyle{remark}
\newtheorem{zau}[theo]{Remark}

\newcommand{\R}{\mathbb{R}}
\newcommand{\Expect}{\mathsf{E}}
\newcommand{\Prob}{\mathsf{P}}

\newcommand{\sas}{S$\alpha$S}

\newcommand{\norm}[1]{\left\lVert#1\right\rVert}
\newcommand{\ind}[1]{\mathbh1_{#1}}

\startlocaldefs

\allowdisplaybreaks
\endlocaldefs

\begin{document}
\begin{frontmatter}

\title{Stochastic wave equation in a plane driven by~spatial~stable~noise}

\author{\inits{L.}\fnm{Larysa}\snm{Pryhara}}\email{pruhara7@gmail.com}
\author{\inits{G.}\fnm{Georgiy}\snm{Shevchenko}\corref{cor1}}\email
{zhora@univ.kiev.ua}
\cortext[cor1]{Corresponding author.}
\address{Mechanics and Mathematics Faculty,\\
Taras Shevchenko National University of Kyiv,\\
Volodymyrska 64/11, 01601 Kyiv, Ukraine}


\markboth{L. Pryhara, G. Shevchenko}{Stochastic wave equation in a plane driven by spatial stable noise}

\begin{abstract}
The main object of this paper is the planar wave equation
\[
\bigg(\frac{\partial^{2}}{\partial t^2}- a^2 \varDelta\bigg)U(x,t)=f(x,t), \quad t\ge0,~x\in\R^2,
\]
with random source $f$. The latter is, in certain sense, a symmetric
$\alpha$-stable spatial white noise multiplied by some regular function
$\sigma$. We define a candidate solution $U$ to the equation via
Poisson's formula and prove that the corresponding expression is well
defined at each point almost surely, although the exceptional set may
depend on the particular point $(x,t)$. We further show that $U$ is H\"older
continuous in time but with probability 1 is unbounded in any
neighborhood of each point where $\sigma$ does not vanish. Finally, we
prove that $U$ is a generalized solution to the equation.
\end{abstract}



%
\received{14 October 2016}
%
\revised{20 October 2016}
%
\accepted{20 October 2016}
\publishedonline{8 November 2016}

\end{frontmatter}

\section*{Introduction}

Stochastic partial differential equations are widely used in modeling
different phenomena involving randomness, and the area of their
application is constantly increasing. This is reflected by increasing
number of works devoted to such equations. Vast majority of these
articles is devoted to the case where the underlying noise is Gaussian.
In particular, a stochastic wave equation with Gaussian noise was
studied in \cite
{balan-tudor-wave,dalang-frangos,dalang-sanz,millet-morien,quer-tindel,walsh},
to mention only a few authors. However, many phenomena are
characterized by heavy tails of the corresponding distributions; often
not only variances, but also expectations of underlying random
variables are infinite. In such cases, the underlying random noise is
better modeled by a stable distribution.

In this paper, we study a wave equation in the plane, where the random
source has a stable distribution. We prove that a candidate solution to
the equation, constructed by means of Poisson's formula, is a
generalized solution. We also show that it is H\"older continuous is
time variable, but it is irregular in the spatial variable.

The paper is organized as follows. Section~\ref{sec:prelim} contains
the notation and auxiliary information on objects involved. In
Section~\ref{sec:main}, we introduce the main object of the paper, a
planar wave equation with stable noise, and establish main results. The
existence and spatial properties of a candidate solution to the
equation, constructed via Poisson's formula, are studied in
Section~\ref{subsec:existence}. In Section~\ref{subsec:generalized}, we prove that the candidate solution is a
generalized solution to the equation. Finally, in Section~\ref{subsec:holder}, we establish the H\"older regularity of the solution
in the time variable.

\section{Preliminaries}\label{sec:prelim}
\subsection{Notational conventions}
Throughout the article, the symbol $C$ denotes a generic constant, the
exact value of which is not important and may change from line to line.
Similarly, $C(\omega)$ is be used to denote a generic a.s.\ finite
random variable. We use the notation $|x|$ both for the absolute value
of a real number and for the Euclidean norm of a vector; the particular
meaning will be always clear from the context. The Euclidean ball $\{
y:|y-x|\le r\}$ is denoted by $B(x,r)$. Finally, $\R_+ = [0,+\infty)$.

\subsection{Stable random variables and measures}

In this section, we give essential information on symmetric $\alpha
$-stable \sas\ random variables and measures; for details, we refer the
reader to \cite{samor-taqqu}.

Let $(\varOmega,\mathcal F,\Prob)$ be a complete probability
space. For a number $\alpha\in(0,2)$, called the stability parameter, a
random variable $\xi$ is \sas\ with the scale parameter $\sigma^{\alpha
}$, $\sigma\ge0$, if its characteristic function is
\[
\Expect\big[e^{i\lambda\xi}\big]=e^{-\sigma^\alpha\left|\lambda
\right|^{\alpha}}.
\]
We also denote $\norm{\xi}_\alpha= \sigma$; note that this is a
(quasi-)norm for $\alpha\ge1$.

\sas\ random variables and fields are often constructed by means of
an~independently scattered \sas\ random measure, which is defined as
follows. Denoting by $\mathcal B_f(\R^d)$ the family of Borel sets of
finite Lebesgue measure, a random set function $M\colon\mathcal B_f(\R
^d)\times\varOmega\to\R$ is called an independently scattered \sas\
random measure with Lebesgue control measure if
\begin{enumerate}[1)]
\item for any $A\in\mathcal B_f(\R^d)$, the random variable $M(A)$ is
\sas\ with scale parameter equal to $\lambda_d(A)$, the Lebesgue
measure of $A$;
\item for any disjoint $A_1, \dots,A_n \in\mathcal B_f(\R^d)$, the
values $M(A_1),\dots,M(A_n)$ are independent.
\item for any disjoint $A_n \in\mathcal B_f(\R^d)$, $n\ge1$, such
that $A = \bigcup_{n=1}^\infty A_n\in\mathcal B_f(\R^d)$,
\[
M(A) = \sum_{n=1}^\infty M(A_n)
\]
almost surely.
\end{enumerate}
For a function $f (x) \in L^\alpha(\R^d)$, the integral
\[
I(f)=\idotsint_{\R^d}f(x)M(dx)
\]
is defined as the limit in probability of integrals of simple compactly
supported functions; its value is an \sas\ random variable with
\[
\big\|I(f)\big\|^\alpha_\alpha=\idotsint_{\R^d}\big|f(x)\big|^\alpha\, dx.
\]

Our analysis is based on the LePage series representation of $M$
defined as follows. Let $\varphi$ be an arbitrary continuous positive
probability density function on $\R^d$, and $\{\varGamma_{k},k\geq1\}, \{
\xi_{k},k\geq1\}, \{g_k,k\geq1\}$ be three independent families of
random variables satisfying:
\begin{itemize}
\item${\varGamma_{k},k\geq1}$, is a sequence of arrivals of a Poisson
process with unit intensity;
\item${\xi_{k},k\geq1}$, are independent random vectors in $\R^d$ with
density $\varphi$;
\item${g_k,k\geq1}$, are independent centered Gaussian variables with
$\Expect[|g_k|^\alpha]=1$.
\end{itemize}

Then $M$ (as a random process indexed by finite measure Borel sets) has
the same finite-dimensional distributions as
\begin{equation}\label{eqLP}
M'(A)=C_\alpha\sum_{k\geq1}\varGamma_k^{-{1}/{\alpha}}\varphi(\xi
_k)^{-1/\alpha} \ind{A}(\xi_k)g_k,
\end{equation}
where $C_\alpha=\big(\frac{\varGamma(2-\alpha)\cos\frac{\pi\alpha
}{2}}{1-\alpha}\big)^{1/\alpha}$; the corresponding series converges
almost surely for any Borel set $A\subset\R^d$ of finite Lebesgue
measure. Moreover, for any $f_1,f_2,\dots,\break f_n \in L^\alpha(\R^d)$,
the vector $(I(f_1), I(f_2), \dots, I(f_n))$ has the same
distribution as $(I'(f_1), I'(f_2), \dots, I'(f_n))$, where
\begin{equation}\label{eqLPI}
I'(f)=C_\alpha\sum_{k\geq1}\varGamma_k^{-{1}/{\alpha}}\varphi(\xi
_k)^{-1/\alpha} f(\xi_k)g_k.
\end{equation}
Throughout the paper, we work\vadjust{\eject} with a planar \sas\ measure $M$, that
is,\ we consider the case $d=2$. We will assume, without loss of
generality, that $M$ is given by \eqref{eqLP}, so that, for any
function $f \in L^\alpha(\R^2)$, the integral
\[
I(f)=\iint_{\R^2}f(x)M(dx)
\]
is given by an almost surely convergent series \eqref{eqLPI}. Moreover,
we assume that
\[
\left(\varOmega, \mathcal F, \Prob\right)=(\varOmega_{\varGamma}\otimes
\varOmega_{\xi} \otimes\varOmega_{g}, \mathcal F_{\varGamma} \otimes\mathcal
F_{\xi} \otimes\mathcal F_{g}, \Prob_{\varGamma}\otimes\Prob_{\xi}
\otimes\Prob_{g} )
\]
and, for all $\omega= (\omega_{\varGamma},\omega_{\xi}, \omega {g})$
and
$k\geq1$,
$\varGamma_{k}(\omega)=\varGamma_{k}(\omega_{\varGamma})$,
$\xi_{k}(\omega)=\xi_{k}(\omega_{\xi})$,
and
$g_{k}(\omega)=g_{k}(\omega_{g})$.
This will not harm the generality but will considerably simplify our exposition.

\section{Main results}\label{sec:main}

For a positive constant $a>0$, consider the planar wave equation
\begin{equation}\label{eqR}
\bigg(\frac{\partial^{2}}{\partial t^2}- a^2 \varDelta\bigg)U
(x,t)=\sigma(x,t)\dot M(x)
\end{equation}
with zero initial conditions. The random source is a product of a
continuous function $\sigma$ and \sas\ white noise $\dot M(x)$, which
is a formal derivative of a planar \sas\ random measure $M$ introduced
in the previous section. The precise meaning of this equality is not
immediately obvious. Clearly, there can be no classical (belonging to
$C^2(\R^2\times\R_+)$ solution to this equation, so we will look at
generalized solutions.

Let $\mathcal D(\R^2\times\R_+)$ denote the class of all compactly
supported infinitely continuously differentiable functions on $\R
^2\times\R_+$. By a generalized solution we mean a function satisfying
\begin{align}
&\int_0^{\infty}\iint_{\R^2}U(x,t)\bigg(\frac{\partial^2}{\partial t^2}\theta(x,t)-a^2\varDelta\theta(x,t)\bigg)\,dx\,dt\nonumber\\
&\quad =\int_0^{\infty}\iint_{\R^2}\theta(x,t)\sigma(x,t)M(dx)\,dt\label{uv}
\end{align}
for all $\theta\in\mathcal D(\R^2\times\R_+)$.

Our approach is to consider a candidate solution given by Poisson's formula
\begin{align}
U(x,t)&=U(x_1,x_2,t)\nonumber\\
&=\frac{1}{2\pi a}\int_0^t\iint_{B(x,a(t-\tau))}\frac{\sigma(y_1,y_2,\tau)M(dy_1,dy_2)\,d\tau}{\sqrt{a^2(t-\tau)^2-(y_1-x_1)^2-(y_2-x_2)^2}}\nonumber\\
&= \frac{1}{2\pi a}\int_0^t\iint_{B(x,a(t-\tau))}\frac{\sigma(y,\tau)M(dy)\,d\tau}{\sqrt{a^2(t-\tau)^2-|y-x|^2}}\label{eqd2a}
\end{align}
and later, in Section~\ref{subsec:generalized}, to show that it
solves \xch{Eq.}{equation}~\eqref{eqR} in a generalized
sense.\vadjust{\eject}

The integral in \eqref{eqd2a} is understood in the following sense: we define
\[
G(x,y,t)=\frac{1}{2\pi a}\int_0^{t-\frac{\left|x-y\right|}{a}}\frac
{\sigma(y,\tau)}{\sqrt{a^2(t-\tau)^2-|y-x|^2}}\,d\tau\,\ind{|x-y|<at}
\]
and set
\begin{equation}\label{eq-U}
U(x,t) = \iint_{\R^2}G(x,y,t)M(dy).
\end{equation}

In what follows, we need some assumptions about the coefficient $\sigma$.
\begin{enumerate}[(S1)]
\item Boundedness:  $|\sigma(x,t)|\le C$ for all $t\ge0$ and $x\in\R^2$.
\item Continuity: $\sigma\in C(\R^2\times\R_+)$.
\item H\"older continuity in time: there exists $\gamma\in(0,1]$ such
that, for all $t,s\ge0$ and $x\in\R^2$,
\[
\big|\sigma(x,t) - \sigma(x,s)\big|\le|t-s|^\gamma.
\]
\end{enumerate}

\subsection{Existence and spatial properties of a candidate
solution}\label{subsec:existence}
First, we establish a result on the existence of the integral defining
the candidate solution $U(x,t)$.
\begin{theo}\label{ex1}
Under {\rm(S1)}, for any $t\ge0$ and $x\in\R^2$, the integral in
\eqref{eq-U} is well defined.
\end{theo}
\begin{proof}
According to the definition of the integral with respect to $M$, the
integral is well defined, provided that
\begin{equation}\label{eqd2u}
\iint_{\R^2}\big|G(x,y,t)\big|^{\alpha}\,dy <\infty.
\end{equation}
Taking into account (S1), we have the estimate
\begin{align*}
\big|G(x,y,t)\big|&=\frac{1}{2\pi a}\bigg|\int_0^{t-\frac{\left|x-y\right|}{a}}\frac{\sigma(y,\tau)}{\sqrt{a^2(t-\tau)^2-|y-x|^2}}\,d\tau\bigg|\,\ind{|x-y|<at}\\
&\leq C\int_0^{t-\frac{\left|x-y\right|}{a}}\frac{d\tau}{\sqrt{a^2(t-\tau)^2-|y-x|^2}}\,\ind{|x-y|<at}\\
&=\frac{C}{a}\ln\bigg(\frac{at}{\left|x-y\right|}+\sqrt{\frac{a^2t^2}{\left|x-y\right|^2}-1}\bigg)\,\ind{|x-y|<at}.
\end{align*}
Therefore,
\begin{align*}
&\iint_{\R^2}\big|G(x,y,t)\big|^{\alpha}\,dy 
\le C \iint_{B(x,at)}\bigg|\ln\bigg(\frac{at}{\left|x-y\right|}+\sqrt
{\frac{a^2t^2}{\left|x-y\right|^2}-1}\bigg)\bigg|^{\alpha}\,dy\\
&\quad =
\begin{vmatrix}
y_1=x_1+at\,r\cos\phi\\
y_2=x_2+at\, r\sin\phi \\ %
\end{vmatrix}
= C\int_0^{2\pi}d\phi\int_0^1r\big(\ln\big|r^{-1}+\sqrt{r^{-2}-1}\big|\big)^{\alpha}\,dr\\
&\quad \le C\int_0^1r\big|\ln\big(2r^{-1}\big)\big|^{\alpha}\,dr\le C\int_0^1 r^{1-\varepsilon}\, dr<\infty,
\end{align*}
where $\varepsilon$ is a small positive number. This proves the statement.
\end{proof}

Recall that $M$ is assumed to coincide with its LePage series, so we
have that, for all $t\ge0$ and $x\in\R^2$, $U(x,t)$ is given by the
almost surely convergent series
\[
U(x,t)=\sum_{k\geq1}\varGamma_k^{-{1}/{\alpha}}\varphi(\xi_k)^{-{1}/{\alpha
}}G(t,x,\xi_k)g_k.
\]

We will further see that the exceptional event of zero probability
generally depends on $x$ and $t$. Moreover, if $\sigma$ is continuous,
then $U$ is unbounded in any neighborhood of any point where $\sigma$
does not vanish. In order to prove this, we first note that $G(x,x,t)$
is infinite for any $t\ge0$ and $x\in\R^2$ such that $\sigma(x,t)\neq
0$. Indeed, let $\sigma(x,t) >0$ for some $t\ge0$ and $x\in\R^2$. Then
there is $\varepsilon>0$ such that $\sigma(x,s) >\varepsilon$ for all
$s\in[t-\varepsilon,t]$. Write
\[
G(x,x,t)=\int_0^{t-\varepsilon}\frac{\sigma(x,\tau)}{a(t-\tau)}\,d\tau
+\int_{t-\varepsilon}^t\frac{\sigma(x,\tau)}{a(t-\tau)}\,d\tau.
\]
The first integral is finite, whereas
\[
\int_{t-\varepsilon}^t\frac{\sigma(x,\tau)}{a(t-\tau)}\,d\tau\geq
\varepsilon\int_{t-\varepsilon}^t\frac{d\tau}{a(t-\tau)}=+\infty.
\]
This observation leads to the following statement.

\begin{theo}\label{teox}
Assume {\rm(S1)} and {\rm(S2)}. Then, for all $t\geq0$ and $x\in\R
^{2}$ such that $\sigma(t,x)\neq0$ and for all $\delta>0$,
\[
\sup_{y\in B(x,\delta)}
\big|U(y,t)\big|=+\infty
\]
almost surely.
\end{theo}
\begin{proof}
Define
\[
\widetilde\varOmega_\xi=\big\{\omega_\xi\in\varOmega_\xi\mid\exists k\geq1:
\big|\xi_k(\omega_\xi)-x\big|\le\delta\big\}.
\]
Since $\{\xi_k, k\ge1\}$ are iid with everywhere positive density, it
is clear that $\Prob_\xi(\widetilde\varOmega_\xi)=1$. Fix some $\omega
_{\xi}\in\widetilde\varOmega_\xi$ and $\omega_{\varGamma} \in\varOmega_{\varGamma
}$. Then $U(x,t)$ has a centered Gaussian distribution, so that by the
0--1 law for Gaussian measures
\[
\Prob_g\Big(\sup_{y\in B(x,\delta)}
\big|U(y,t)\big|<+\infty\Big)\in\left\{0,1\right\}.
\]
Suppose by contradiction that
\[
\Prob_g\Big(\sup_{y\in B(x,\delta)}
\big|U(y,t)\big|<+\infty\Big)=1.
\]
Then by Fernique's theorem
\begin{equation}\label{eq-U-finvar}
\Expect_g\Big[\sup_{y\in B(x,\delta)}
\big|U(y,t)\big|^2\Big]<\infty.
\end{equation}
On the other hand,
\begin{align*}
\Expect_g\Big[\sup_{y\in B(x,\delta)}\big|U(y,t)\big|^2\Big]&\geq\sup_{y\in B(x,\delta)} \Expect_g \big[\big|U(y,t)\big|^2\big]\\
&=C_\alpha^2\sup_{y\in B(x,\delta)}\sum_{k\geq1}\varGamma_k^{-2/\alpha}\varphi(\xi_k)^{-{2}/{\alpha}}\big|G(y,\xi_k,t)\big|^2\\
&\geq C_\alpha^2\varGamma_{k(\omega_\xi)}^{-2/\alpha}\varphi(\xi_{k(\omega_\xi)})^{-{2}/{\alpha}}\sup_{y\in B(x,\delta)} G(y,\xi_{k(\omega_\xi)},t)^2,
\end{align*}
where $k(\omega_\xi)$ is an integer such that $|\xi_{k(\omega_\xi
)}-x|<\delta$ (it exists since $\omega_{\xi}\in\widetilde\varOmega
_\xi$).

Since $\sigma$ is continuous, there exists $\varepsilon>0$ such that,
for all $y\in\R^2$ with $|x-y|<\varepsilon$, $\sigma
(y,t)\neq0$. Without loss of generality, we can assume that
$\varepsilon\geq\delta$. Taking into account that $\xi_{k(\omega_\xi
)}\in B(x,\delta)$, we get
\[
\Expect_g\Big[\sup_{y\in B(x,\delta)}
\big|U(y,t)\big|^2\Big]\geq C_\alpha^2 \varGamma_{k(\omega_\xi
)}^{-2/\alpha}\varphi(\xi_{k(\omega_\xi)})^{-2/\alpha}G(\xi
_{k(\omega_\xi)},\xi_{k(\omega_\xi)},t)^2.
\]
However, since $\sigma(\xi_{k(\omega_\xi)},t)\neq0$, the observation
preceding the theorem yields\break $G(\xi_{k(\omega_\xi)},\xi_{k(\omega
_\xi)},t)^2=+\infty$, which contradicts \eqref{eq-U-finvar}.

Consequently, for all $\omega_{\xi}\in\widetilde\varOmega_\xi$ and
$\omega_{\varGamma} \in\varOmega_{\varGamma}$,
\[
\Prob_g\Big(\sup_{y\in B(x,\delta)}
\big|U(y,t)\big|<+\infty\Big)=0,
\]
whence
\begin{align*}
&\Prob\Big(\sup_{y\in B(x,\delta)}\big|U(y,t)\big|<+\infty\Big)\\
&\quad =\int_{\varOmega_{\varGamma}}\int_{\varOmega_{\xi}}\Prob_g\Big(\sup_{y\in B(x,\delta)}\big|U(y,t)\big|<+\infty\Big)\,d\Prob_\xi(\omega_\xi)\,d\Prob_\varGamma(\omega_\varGamma)=0,
\end{align*}
as claimed.
\end{proof}

\subsection{Generalized solution}\label{subsec:generalized}
Theorem \ref{teox} shows that the function $U(x,t)$ cannot be a
classical solution to Eq.~\eqref{eqR}. Our next aim is to show that it
solves \eqref{eqR} in a generalized sense.

\begin{theo}\label{ur} Assume {\rm(S1)} and {\rm(S2)}.
1. If $\alpha\in(0,1)$, then there is $\varOmega_0\in\mathcal F$, ${\Prob
(\varOmega_0)=1}$, such that, for all $\omega\in\varOmega_0$ and $\theta\in
\mathcal D(\R^2\times\R_+)$, Eq.~\eqref{uv} holds.
2. If $\alpha\in[1,2)$, then, for any $\theta\in\mathcal D(\R
^2\times\R_+)$, Eq.~\eqref{uv} holds almost surely.
\end{theo}
\begin{zau}\label{a}
In the second part of this theorem, the exceptional event of
probability zero may depend on $\theta$.
\end{zau}
\begin{proof}

Write the LePage representation for the left-hand side of Eq.~\eqref{uv}:
\begin{gather*}
L(\theta)=C_\alpha\sum_{k=1}^\infty\varGamma_k^{-{1}/{\alpha}}K(x,\xi
_k,t)g_k,\\
K(x,y,t)=\varphi(\xi_k)^{-1/\alpha}\int_0^{\infty}\iint_{\R^2}
G(x,y,t)\Big(\frac{\partial^2}{\partial t^2}\theta(x,t)-a^2\varDelta
\theta(x,t)\Big)\, dx\,dt
\end{gather*}
and its right-hand side
\begin{gather*}
R(\theta)=C_\alpha\int_0^{\infty}\sum_{k=1}^\infty\varGamma
_k^{-{1}/{\alpha}}\varphi(\xi_k)^{-1/\alpha}\theta(\xi_k,t)\sigma(\xi
_k,t)g_k\, dt.
\end{gather*}
The proof consists of two steps: showing the convergence of the LePage
series and then proving that \eqref{uv} holds for partial sums of the
LePage series.

Let us estimate the terms in the series for $L(\theta)$. Assume that
$\operatorname{supp} \theta\subset[0,R]\times B(0,R)$. Then, denoting
$\psi(x,t) = \frac{\partial^2}{\partial t^2}\theta(x,t)-a^2\varDelta
\theta(x,t)$, we have
%
%
\begin{align}\label{k}
\big|K(x,\xi_k,t)\big|&\,{\leq}\int_0^R\iint_{B(0,R)}\big|\varphi(\xi_k)\big|^{-1/\alpha}\big|G(x,\xi_k,t)\big|\big|\psi(x,t)\big|\,dx\,dt\nonumber\\
&\,{\leq}\Big(\inf_{B(0,R)}\big|\varphi(x)\big|\Big)^{-1/\alpha}\nonumber\\
& \quad \times \sup_{x\in\R^2,t\ge0}\big|\psi(x,t)\big| \int_0^R\!\iint_{B(0,R)}\big|G(x,\xi_k,t)\big|\,dx\,dt\nonumber\\
&\,{\leq}\, C_{R,\theta}.
\end{align}
%

Consider first the case $\alpha\in(0,1)$.
By the strong law of large numbers and well-known properties of
Gaussian random variables there exists $\varOmega_0\in\mathcal F, \Prob
(\varOmega_0)=1$, such that, for all $\omega\in\varOmega_0$ and $k\geq1$,
\[
\varGamma_k\geq C_1(\omega)k,\qquad\left|g_k\right|\leq C_2(\omega)\left
(\log k+1\right),
\]
where $C_1, C_2$ are some positive random variables.
Therefore, the $k$th term in the series for $L(\theta)$ is bounded by
\begin{equation*} 
C_\alpha\varGamma_k^{-1/\alpha} \big|K(t,x,\xi_k)g_k\big|\leq C(\omega
)k^{-1/\alpha} \left|\log k+1\right|.
\end{equation*}
Consequently, the series for $L(\theta)$ is convergent for all $\omega
\in\varOmega_0$ and $\theta(x,t)\in\mathcal D(\R^2\times\R_{+})$.
Similarly, we can show the convergence of $R(\theta)$.

For $\alpha\in[1,2)$, the argument is changed slightly. Specifically,
we show the almost sure convergence of the series for $L(\theta)$ and
$R(\theta)$ for all $\theta(x,t)\in\mathcal D(\R^2\times\R_{+})$.
Indeed, for fixed $\omega_\xi\in\varOmega_\xi, \omega_\varGamma\in\varOmega
_\varGamma$, in view of \eqref{k},
we have
\[
\Expect_g\big[L(\theta)^2\big]=C_\alpha^2\sum_{k\geq1}\varGamma
_k^{-2/\alpha}K(x,\xi_k,t)^2\le C(\omega)\sum_{k\geq1}k^{-2/\alpha}
\]
almost surely. Therefore, by the Kolmogorov theorem the series for
$L(\theta)$ converges $\Prob_g$-almost surely for almost all $\omega_\xi
\in\varOmega_\xi$ and $\omega_\varGamma\in\varOmega_\varGamma$ and, therefore,
$\Prob$-almost surely. The almost sure convergence of $R(\theta)$ is
shown in a similar way.

Now we prove that Eq.~\eqref{uv} holds for partial sums of the LePage
series; the argument does not depend on the value of $\alpha$.
The counterpart of Eq.~\eqref{uv} for the partial sums reads
as\goodbreak
\begin{align*}
&C_\alpha\sum_{k=1}^N\varGamma_k^{-{1}/{\alpha}}\varphi(\xi_k)^{-1/\alpha}\\
&\qquad \times\int_0^{\infty}\iint_{\R^2}\int_0^t\frac{\sigma(\xi_k, \tau)}{\sqrt{a^2(t-\tau)^2-\left|y-\xi_k\right|^2}}\ind{\left|y-\xi_k\right|<a(t-\tau)}\psi(y,t)\,d\tau\, dy\,dt\\[2pt]
&\quad =C_\alpha\sum_{k=1}^N\varGamma_k^{-{1}/{\alpha}}\varphi(\xi_k)^{-1/\alpha}\int_0^{\infty}\theta(\xi_k, \tau)\sigma(\xi_k, \tau)\,d\tau.
\end{align*}
It suffices to show the equality of the corresponding terms, that is,\
to prove that, for any $x\in\R^2$,
\[
\int_0^{\infty}\int_\tau^\infty\iint_{B(x,a(t-\tau))}\frac{\sigma(y,\tau
)\psi(y,t)}{\sqrt{a^2(t-\tau)^2-\left|x-y\right|^2}}\,dy\,dt\,d\tau=\int
_0^\infty\sigma(x,\tau)\theta(x,\tau)\,d\tau.
\]
This equality, in turn, would follow if we show that
\begin{equation}\label{psi}
\int_\tau^\infty\iint_{B(x,a(t-\tau))}\frac{\psi(y,t)}{\sqrt{a^2(t-\tau
)^2-\left|x-y\right|^2}}\,dy\,dt=\theta(x,\tau)
\end{equation}
for all $\tau\geq0, y\in\R^2$.

As before, assume that $\operatorname{supp} \theta\subset[0,R]\times B(0,R)$.
Define
\[
\widetilde{\theta}(x,u)=\theta(x,R-u),\quad u\leq R.
\]
Then
\begin{gather*}
\frac{\partial^2}{\partial t^2}\theta(x,R-u)=\frac{\partial^2}{\partial t^2}\widetilde{\theta}(x,u);
\qquad
\widetilde{\theta}(x,0)=0;
\qquad
\frac{\partial}{\partial t}\widetilde{\theta}(x,0)=0;\\
\varDelta\theta(x,R-u)=\varDelta\widetilde{\theta}(x,u).
\end{gather*}
Consider the following Cauchy problem:
\begin{gather*}
\left\{
\begin{aligned}
&\bigg(\frac{\partial^{2}}{\partial t^2}- a^2 \varDelta\bigg)V\left
(x,t\right)=\frac{\partial^2}{\partial t^2}\widetilde\theta
(x,t)-a^2\varDelta\widetilde\theta(x,t),\\
&V(x,0)=0, \\
&\frac{\partial V(x,t)}{\partial t}|_{t=0}= 0,\\
\end{aligned}
\right.
\end{gather*}
Clearly, the function $\widetilde\theta$ is a solution. On the other
hand, by Poisson's formula, for all $x\in\R^2$ and $r\ge0$,
\[
\widetilde\theta(x,r)=\frac{1}{2\pi a}\int_0^r\iint_{B(x,a(r-s))}\frac
{\frac{\partial^2}{\partial t^2}\widetilde\theta(y,s)-a^2\varDelta
\widetilde\theta(y,s)}{\sqrt{a^2(r-s)^2-\left|x-y\right|^2}}\,dy\,ds.
\]
Changing the variables $r\to R-\tau$, $s\to R-t$ and noticing that $\psi
(t,x)$ vanishes for $\tau\ge R$, we get \eqref{psi} for all $t\in
[0,R]$. For $\tau\ge R$, the both sides of the equality are zero,
whence the proof follows.
\end{proof}

\subsection{Regularity of solution in time variable}\label{subsec:holder}
In this section, adapting the argument of \cite{kono-maejima}, we show
that the solution $U$ constructed by means of Poisson's formula \eqref
{eq-U} is H\"older continuous in the time variable. Since we have
already shown that $U$ is highly irregular in the spatial variable, our
findings for the planar wave equation are in a sharp contrast with the
scalar case, where the time and space regularity are the same.
\begin{theo}\label{teog1}
Assume {\rm(S1)--\rm(S3)}. Then for any $T>0$ and $x\in\R^2$ the
function $U(x,\cdot)$ is H\"older continuous on $[0,T]$ almost surely
with any exponent less than $\gamma\wedge\frac12$. This implies the
required statement.
\end{theo}

\begin{proof}

Take some $h>0$ and $t\in[0,T-h]$. For fixed $\omega_\xi\in\varOmega_\xi$
and $\omega_\varGamma\in\varOmega_\varGamma$, we have the estimate
\begin{align*}
&\Expect_g\big[\big(U_1(x,t)-U_1(x,t+h)\big)^2\big]\\
&\quad =C_\alpha^2\sum_{k\geq1}\varGamma_k^{-2/\alpha}\varphi(\xi_k)^{-2/\alpha}\big|G(x,\xi_k,t_2)-G(x,\xi_k,t_2)\big|^2\leq g(h),
\end{align*}
where\vspace{-6pt}
\begin{gather*}
g(h)=C_\alpha^2\sum_{k\geq1}\varGamma_k^{-2/\alpha}\varphi
(\xi_k)^{-2/\alpha}\sup_{\substack{
t_1, t_2 \in[0,T] \\
0<t_1-t_2<h
}}
\big|G(x,\xi_k,t_1)-G(x,\xi_k,t_2)\big|^2.
\end{gather*}
Let $t_1,t_2\in[0,T]$ be such that $0<t_2-t_1<h$. Write
\begin{align*}
&\big|G(x,y,t_1)-G(x,y,t_2)\big|^2\\
&\quad =\bigg|\int_0^{t_1-\frac{\left|x-y\right|}{a}}\frac{\sigma(y,\tau)}{\sqrt{a^2(t_1-\tau)^2-|x-y|^2}}\,d\tau\,\ind{\left|x-y\right|<at_1}\\[2pt]
&\qquad -\int_0^{t_2-\frac{\left|x-y\right|}{a}}\frac{\sigma(y,\tau)}{\sqrt{a^2(t_2-\tau)^2-|x-y|^2}}\,d\tau\,\ind{\left|x-y\right|<at_2}\bigg|^2.
\end{align*}
Consider first the case $|x-y|<at_1$. Changes of variable
$\tau=t_1-s$ and $\tau=t_2-s$ in the first and second integrals,
respectively, give
\begin{align*}
&\bigg|\int_{\frac{\left|x-y\right|}{a}}^{t_1}\frac{\sigma(y,t_1-s)\,ds}{\sqrt{a^2s^2-|x-y|^2}}-\int_{\frac{\left|x-y\right|}{a}}^{t_2}\frac{\sigma(y,t_2-s)\,ds}{\sqrt{a^2s^2-|x-y|^2}}\bigg|\\[4pt]
&\quad \leq\bigg|\int_{\frac{\left|x-y\right|}{a}}^{t_1}\frac{\sigma(y,t_1-s)-\sigma(y,t_2-s)}{\sqrt{a^2s^2-|x-y|^2}}\,ds\bigg|+\bigg|\int_{t_1}^{t_2}\frac{\sigma(y,t_1-s)\,ds}{\sqrt{a^2s^2-|x-y|^2}}\bigg|=:I_1+I_2.
\end{align*}
Taking into account (S3), we have
\begin{align*}
I_1&\leq\int_{\frac{\left|x-y\right|}{a}}^{t_1}\frac{\left|\sigma(y,t_1-s)-\sigma(y,t_2-s)\right|}{\sqrt{a^2s^2-|x-y|^2}}\,ds\\
&\leq\left|t_1-t_2\right|^\gamma\int_{\frac{\left|x-y\right|}{a}}^{t_1}\frac{ds}{\sqrt{a^2s^2-\left|x-y\right|^2}}\\
&\leq C\left|t_1-t_2\right|^\gamma\int_{\frac{\left|x-y\right|}{a}}^{t_1}\frac{ds}{\sqrt{(as-\left|x-y\right|)(as+\left|x-y\right|)}} \\
&\leq\frac{C\left|t_1-t_2\right|^\gamma\sqrt{at_1-\left|x-y\right|}}{\sqrt{\left|x-y\right|}}\leq\frac{C\left|t_1-t_2\right|^{\gamma}}{\sqrt{\left|x-y\right|}};\\
I_2&\leq C\int_{t_1}^{t_2}\frac{ds}{\sqrt{a^2s^2-\left|x-y\right|^2}}\leq C\int_{t_1}^{t_2}\xch{\frac{ds}{\sqrt{\left(as-\left|x-y\right|\right)\left(as+\left|x-y\right|\right)}}}{\frac{ds}{\sqrt{\left(as-\left|x-y\right|\right)\left(as+\left|x-y\right|\right)}}\leq}\\
&\leq\frac{2C}{a\sqrt{\left|x-y\right|}}\big(\sqrt{at_2-\left|x-y\right|}-\sqrt{at_1-\left|x-y\right|}\big)\leq\frac{C\sqrt{t_2-t_1}}{\sqrt{\left|x-y\right|}}.
\end{align*}
Now let $at_1<|x-y|<at_2$. In this case,
\begin{align*}
&\big|G(t_1,x,y)-G(t_2,x,y)\big| \\
&\quad =\bigg|\int_0^{t_2-\frac{\left|x-y\right|}{a}}\frac{\sigma(y,\tau)}{\sqrt{a^2(t_2-\tau)^2-|x-y|^2}}\,d\tau\ind{\left|x-y\right|<at_2}\bigg| \\
&\quad \leq C \int_0^{t_2-\frac{\left|x-y\right|}{a}}\frac{ \ind{\left|x-y\right|<aT}\, d\tau}{\sqrt{a^2(t_2-\tau)^2-\left|x-y\right|^2}}\\
&\quad \leq C \int_0^{t_2-\frac{\left|x-y\right|}{a}}\frac{\ind{\left|x-y\right|<aT}\,d\tau}{\sqrt{(a(t_2-\tau)-\left|x-y\right|)(a(t_2-\tau)+\left|x-y\right|)}}\\
&\quad \leq\frac{C}{\sqrt{\left|x-y\right|}}\int_0^{t_2-\frac{\left|x-y\right|}{a}}\frac{\ind{\left|x-y\right|<aT}\,d\tau}{\sqrt{(a(t_2-\tau)-\left|x-y\right|)}} \\
&\quad \leq\frac{C\sqrt{at_2-\left|x-y\right|}}{\sqrt{\left|x-y\right|}} \ind{\left|x-y\right|<aT} \leq\frac{C\sqrt{h}}{\sqrt{\left|x-y\right|}} \ind{\left|x-y\right|<aT}.
\end{align*}
Combining the estimates, we get
\[
g(h)\leq C h^{2\beta}\sum_{k\geq1}\varGamma_k^{-2/\alpha}\varphi(\xi
_k)^{-2/\alpha}\frac{1}{\left|x-\xi_k\right|}\ind{\left|x-\xi_k\right|<aT},
\]
where $\beta= \gamma\wedge\frac12$.
Taking the expectation w.r.t.\ $\xi$, we obtain
\begin{align*}
\Expect_\xi\big[g(h)\big]&\leq C h^{2\beta}\sum_{k\geq1}\varGamma_k^{-2/\alpha}\Expect_\xi\bigg[\frac{\varphi(\xi_k)^{-2/\alpha}}{\left|x-\xi_k\right|}\ind{\left|x-\xi_k\right|<aT}\bigg]\\
&\leq C h^{2\beta}\sum_{k\geq1}\varGamma_k^{-2/\alpha}\inf_{B(x,aT)}\varphi(y)^{-2/\alpha}\iint_{B(x,aT)}\frac{1}{\left|x-y\right|}\,dy\le C(\omega_\varGamma) h^{2\beta}.
\end{align*}
Hence, for any $\eta>0$,
\[
\Expect_\xi\Bigg[\sum_{n=1}^{\infty}\frac{2^{n\beta}}{n^{1+\eta
}}g\big(2^{-n}\big)\Bigg]\leq C(\omega_\varGamma)\sum_{n=1}^{\infty}n^{-1-\eta
}<\infty.
\]
Therefore,
\[
\sum_{n=1}^{\infty}\frac{2^{n\beta}g(2^{-n})}{n^{1+\eta}}<\infty
\]
$\Prob_\xi\otimes\Prob_\varGamma$-almost surely. In particular, $ \frac
{2^{n\beta}g(2^{-n})}{n^{1+\eta}}\rightarrow0$, $n\rightarrow\infty$,
$\Prob_\xi\otimes\Prob_\varGamma$-almost surely. Since the function $g$ is
nondecreasing and the function $f(x) = x^{-2\beta}|\ln x|^{1+\eta}$
satisfies $f(2x)\le Cf(x)$ for $x$ small enough, the latter convergence implies
\[
\frac{g(h)}{h^{2\beta}\left|\ln h\right|^{1+\eta}}\rightarrow0,\quad
h\rightarrow0,
\]
$\Prob_\xi\otimes\Prob_\varGamma$-almost surely. Therefore, the inequality
\[
\Expect_g\big[\big(U(x,t_1)-U(x,t_2)\big)^2\big]\leq
(t_2-t_1)^{2\beta}\big|\ln(t_2-t_1)\big|^{1+\eta}\xch{,}{.}
\]
holds $\Prob_\xi\otimes\Prob_\varGamma$-almost surely for all $t_1,t_2\in
[0,T]$ close enough and such that $t_1<t_2$. Since $U$ has a centered
Gaussian distribution for fixed $\omega_\xi\in\varOmega_\xi$ and $\omega
_\varGamma\in\varOmega_\varGamma$, the last observation yields that
\[
\big|U(x,t_1)-U(x,t_2)\big|\leq C(\omega) \left|t_1-t_2\right|^{\beta
}\big|\ln\left|t_2-t_1\right|+1\big|^{1+\eta/2}
\]
$\Prob$-almost surely for all $t_1,t_2\in[0,T]$ close enough.
\end{proof}
\bibliographystyle{vmsta-mathphys}
%

\end{document}